\documentclass[12pt]{amsart}

\usepackage{amssymb,latexsym}

\usepackage{enumerate}

\makeatletter

\@namedef{subjclassname@2010}{

  \textup{2010} Mathematics Subject Classification}

\makeatother
\newtheorem{thm}{Theorem}[section]

\newtheorem{con}{Conjecture}[section]

\newtheorem{lem}[thm]{Lemma}
\newtheorem{pro}[thm]{Proposition}
\theoremstyle{definition}

\numberwithin{equation}{section}

\newcommand{\vx}{{\bf x}}
\newcommand{\vt}{{\bf t}}

\newcommand{\va}{{\bf a}}
\newcommand{\pr}{\textup{P}}
\newcommand{\ex}{\mathbb{E}}
\newcommand{\Va}{\textup{Var}(q)}
\newcommand{\re}{\textup{Re}}
\newcommand{\Cov}{\textup{Cov}_{q;a_1,\dots,a_r}}

\begin{document}

\baselineskip=17pt

\title{The Shanks–-R\'enyi prime number race with many contestants}

\author[Youness Lamzouri]{Youness Lamzouri}

\address{Department of Mathematics, University of Illinois at Urbana-Champaign,
1409 W. Green Street,
Urbana, IL, 61801
USA}

\email{lamzouri@math.uiuc.edu}

\date{}

\begin{abstract} Under certain plausible assumptions, M. Rubinstein and P. Sarnak solved the Shanks–-R\'enyi race problem, by showing that the set of real numbers $x\geq 2$ such that $\pi(x;q,a_1)>\pi(x;q,a_2)>\dots>\pi(x;q,a_r)$  has a positive logarithmic density $\delta_{q;a_1,\dots,a_r}$. Furthermore, they established that if $r$ is fixed, $\delta_{q;a_1,\dots,a_r}\to 1/r!$ as $q\to \infty$. In this paper, we investigate the size of these densities when the number of contestants $r$ tends to infinity with $q$. In particular, we deduce a strong form of a recent conjecture of  A. Feuerverger and G. Martin which states that $\delta_{q;a_1,\dots,a_r}=o(1)$ in this case.
Among our results, we prove that $\delta_{q;a_1,\dots,a_r}\sim 1/r!$  in the region $r=o(\sqrt{\log q})$ as $q\to\infty$. We also bound the order of magnitude of these densities beyond this range of $r$. For example, we show that when $\log q\leq r\leq \phi(q)$, $\delta_{q;a_1,\dots,a_r}\ll_{\epsilon} q^{-1+\epsilon}$.

\end{abstract}

\subjclass[2010]{Primary 11N13; Secondary 11N69, 11M26}

\keywords{The Shanks–-R\'enyi race problem, primes in arithmetic progressions, zeros of Dirichlet $L$-functions.}

\thanks{The author is supported by a postdoctoral fellowship from the Natural Sciences and Engineering Research Council of Canada.}

\maketitle

\section{Introduction}
A classical problem in analytic number theory is the so-called ``Shanks--R\'enyi prime number race'' which concerns the distribution of prime numbers in arithmetic progressions. As colorfully described by Knapowski and Tur\'an in \cite{KT}, let $q\geq 3$ and $2\leq r\leq \phi(q)$ be positive integers, and denote by $\mathcal{A}_r(q)$ the set of ordered $r$-tuples of distinct residue classes $(a_1,a_2,\dots, a_r)$ modulo $q$ which are coprime to $q$. For $(a_1,a_2,\dots, a_r)\in \mathcal{A}_r(q)$, consider a game with $r$ players called ``$1$'' through ``$r$'', where at time $x$, the player ``$j$'' has a score of $\pi(x;q,a_j)$ (where $\pi(x;q,a)$ denotes the number of primes $p\leq x$ with $p\equiv a\bmod q$).
As $x\to\infty$, will all $r!$ orderings of the players occur for infinitely many integers $x$?

It is generally believed that the answer to this question is yes for all $q$ and all $(a_1,a_2,\dots, a_r)\in \mathcal{A}_r(q)$. An old result of Littlewood \cite{Li1} shows that this is indeed true in the special cases $(q,a_1,a_2)=(4,1,3)$ and $(q,a_1,a_2)=(3,1,2)$. Since then, this problem has been extensively studied by many authors, including  Knapowski and Tur\'an \cite{KT}, Bays and  Hudson \cite{BH1} and \cite{BH2},  Kaczorowski \cite{Ka1}, \cite{Ka2} and \cite{Ka3}, Feuerverger and Martin \cite{FeM}, Martin \cite{Ma}, Ford and Konyagin \cite{FK1} and \cite{FK2}, Fiorilli and Martin \cite{FiM}, and the author \cite{La1} and \cite{La2}.

A major breakthrough was made in 1994 by Rubinstein and Sarnak who completely solved this problem in \cite{RS}, conditionally on the two following assumptions:
\begin{itemize}
\item The Generalized Riemann Hypothesis (GRH): all nontrivial zeros of Dirichlet $L$-functions have real part equal $1/2$.
 \item The Linear Independence Hypothesis (LI) (also known as the Grand Simplicity Hypothesis): the nonnegative imaginary parts of the nontrivial zeros of Dirichlet $L$-functions attached to primitive characters are linearly independent over $\mathbb{Q}$.
\end{itemize}
Rubinstein and Sarnak proved, under these two hypotheses, the stronger result that for any $(a_1,\dots, a_r)\in \mathcal{A}_r(q)$,  the set of real numbers $x\geq 2$ such that
$$ \pi(x;q,a_1)>\pi(x;q,a_2)>\dots>\pi(x;q,a_r),$$ has a positive logarithmic density, which shall be denoted throughout this paper by $\delta_{q;a_1,\dots,a_r}.$ (Recall that the logarithmic density of a subset $S$ of $\mathbb{R}$ is defined as
$$ \delta_{S}:= \lim_{x\to \infty}\frac{1}{\log x}\int_{t\in S\cap [2,x]}\frac{dt}{t},$$
provided that this limit exists). To establish this result, they constructed an absolutely continuous measure $\mu_{q;a_1,\dots,a_r}$ for which
\begin{equation}
 \delta_{q;a_1,\dots,a_r} =\int_{x_1>x_2>\cdots >x_r}d\mu_{q;a_1,\dots,a_r}(x_1,\dots,x_r).
\end{equation}

Among the results they derived on these densities, Rubinstein and Sarnak showed that in an $r$-way race with $r$ fixed, all biases disappear when $q\to\infty$. More specifically they proved
\begin{equation}
\lim_{q\to\infty}\max_{(a_1,\dots,a_r)\in\mathcal{A}_r(q)}\left|r!\delta_{q;a_1,\dots,a_r}-1\right|=0.
\end{equation}
Recently, Fiorilli and Martin \cite{FiM} established an asymptotic formula for the density in a two-way race, which allows them to determine the exact rate at which $\delta_{q;a_1,a_2}$ converges to $1/2$ as $q$ grows. Shortly after, the author \cite{La1} succeeded to obtain an asymptotic formula for $\delta_{q;a_1,\dots,a_r}$ for any fixed $r\geq 3$ as $q\to\infty$, in which the rate of convergence to $1/r!$ is surprisingly different from the case $r=2$.

However, as far as the author of the present paper knows, no results have been obtained on the size of the densities $\delta_{q;a_1,\dots,a_r}$ if  $r\to\infty$ as $q\to\infty$. In \cite{FeM}, Feuerverger and Martin conjectured that in this case we should have $\delta_{q;a_1,\dots,a_r}=o(1)$. They also asked whether one can prove a uniform version of the result of Rubinstein and Sarnak (1.2), namely that this statement holds in a certain range $r\leq r_0(q)$ for some $r_0(q)\to\infty$ as $q\to\infty$.
\begin{con}[Feuerverger--Martin] We have
$$\lim_{q\to\infty}\max_{(a_1,\dots,a_r)\in\mathcal{A}_r(q)}\delta_{q;a_1,\dots,a_r}=0,$$
for any arbitrary function $r=r(q)$ tending to infinity with $q$.
\end{con}

In the present paper, we investigate the order of magnitude of $\delta_{q;a_1,\dots,a_r}$ when the number of contestants $r\to\infty$ as $q\to\infty$. In particular, answering the question of Feuerverger and Martin, we establish a uniform version of (1.2), and obtain a strong quantitative form of Conjecture 1.1.

\begin{thm} Assume GRH and LI. Let $q$ be a large positive integer. Then, for any integer $r$ such that $2\leq r\leq \sqrt{\log q}$ we have

 $$\delta_{q;a_1,\dots,a_r}= \frac{1}{r!}\left(1+O\left(\frac{r^2}{\log q}\right)\right),$$
 uniformly for all $r$-tuples $(a_1,\dots,a_r)\in\mathcal{A}_r(q)$.
\end{thm}
As a consequence, Theorem 1.1 implies that  (1.2) holds true in the range $r=o(\sqrt{\log q})$ as $q\to\infty$. Indeed in this region of $r$, all biases disappear when $q\to\infty$, namely
\begin{equation}
 \delta_{q;a_1,\dots,a_r}\sim \frac{1}{r!},
\end{equation}
uniformly for all $r$-tuples $(a_1,\dots,a_r)\in\mathcal{A}_r(q)$.
Moreover, one can also deduce that if $c_0>0$ is a suitably small constant and $r\leq c_0\sqrt{\log q}$, then uniformly for all $r$-tuples $(a_1,\dots,a_r)\in\mathcal{A}_r(q)$ we have
\begin{equation}
\delta_{q;a_1,\dots,a_r}\asymp \frac{1}{r!}.
\end{equation}

Note that  $1/r!=\exp(-r\log r+r+O(\log r))$ by Stirling's formula. Our next result shows that the densities $\delta_{q;a_1,\dots, a_r}$ have roughly the same asymptotic decay in the range  $\sqrt{\log q}\ll r\leq (1-\epsilon)\log q/\log\log q$, for any $\epsilon>0$.
\begin{thm} Assume GRH and LI. For any $\epsilon>0$, if $q$ is large and
$\sqrt{\log q}\ll r\leq (1-\epsilon)\log q/\log\log q$ is an integer, then
$$\delta_{q;a_1,\dots,a_r}=\exp\left(-r\log r+r+O\left(\log r+ \frac{r^2}{\log q}\right)\right),$$
uniformly for all $r$-tuples $(a_1,\dots,a_r)\in\mathcal{A}_r(q)$.
\end{thm}
It would be interesting to determine the order of magnitude of the densities $\delta_{q;a_1,\dots,a_r}$ beyond the region $r\leq (1-\epsilon)\log q/\log\log q$. Unfortunately, this range seems to be the limit of what can be achieved using our method. Nevertheless, we can use Theorem 1.2 to obtain an upper bound for $\delta_{q;a_1,\dots,a_r}$ beyond this range of $r$.
\begin{thm} Assume GRH and LI. For any $\epsilon>0$, if $q$ is large and  $(1-\epsilon/2)\log q/\log\log q\leq r \leq \phi(q)$ is an integer, then
$$\max_{(a_1,\dots,a_r)\in\mathcal{A}_r(q)}\delta_{q;a_1,\dots,a_r}\ll_{\epsilon} \frac{1}{q^{1-\epsilon}}.$$
\end{thm}

The paper is organized as follows. In Section 2, following the work of Rubinstein and Sarnak, we shall construct the measure $\mu_{q;a_1,\dots,a_r}$ as a probability distribution corresponding to a certain random vector and study its covariance matrix and large deviations. In Section 3, we investigate the Fourier transform of $\mu_{q;a_1,\dots,a_r}$ and show that in a certain range $\hat{\mu}_{q;a_1,\dots,a_r}$ can be approximated by the Fourier transform of a multivariate normal distribution having the same covariance matrix. In Section 4, we study properties of multivariate normal distributions and prove Theorems 1.1, 1.2 and 1.3.

\section{The measure  $\mu_{q;a_1,\dots,a_r}$}

We begin by developing the necessary notation to construct the measure $\mu_{q;a_1,\dots,a_r}$, following the work of Rubinstein and Sarnak \cite{RS}. For $(a_1,a_2,\dots, a_r)\in \mathcal{A}_r(q)$ we introduce the vector-valued function
$$ E_{q;a_1,\dots,a_r}(x):=(E(x;q,a_1), \dots, E(x;q,a_r)),$$
where
$$ E(x;q,a):=\frac{\log x}{\sqrt{x}}\left(\phi(q)\pi(x;q,a)-\pi(x)\right).$$
The normalization is such that, if we assume GRH, $E_{q;a_1,\dots,a_r}(x)$ varies roughly boundedly as $x$ varies.
Moreover, for a nontrivial character $\chi$ modulo $q$, we denote by $\{\gamma_{\chi}\}$ the sequence of imaginary parts of the nontrivial zeros of  $L(s,\chi)$. Let $\chi_0$ denote the principal character modulo $q$ and define $S= \cup_{\chi\neq \chi_0\bmod q}\{\gamma_{\chi}\}$. Furthermore, let  $\{U(\gamma_{\chi})\}_{\gamma_{\chi}\in S}$ be a sequence of independent random variables uniformly distributed on the unit circle.

Rubinstein and Sarnak established, under GRH and LI, that the vector-valued function $E_{q;a_1,\dots,a_r}$ has a limiting distribution $\mu_{q;a_1,\dots,a_r}$, where $\mu_{q;a_1,\dots,a_r}$ is the probability measure corresponding to the random vector
\begin{equation*}
X_{q;a_1,\dots,a_r}= (X(q,a_1),\dots,X(q,a_r)),
\end{equation*}
where
$$ X(q,a)= -C_q(a)+ \sum_{\substack{\chi\neq \chi_0\\ \chi\bmod q}}\sum_{\gamma_{\chi}>0} \frac{2\text{Re}(\chi(a)U(\gamma_{\chi}))}{\sqrt{\frac14+\gamma_{\chi}^2}},$$
and $$C_q(a):=-1+ \sum_{\substack{b^2\equiv a \bmod q\\ 1\leq b\leq q}}1.$$
Note that for $(a,q)=1$ the function $C_q(a)$ takes only two values: $C_q(a)=-1$ if $a$ is a non-square modulo $q$, and $C_q(a)=C_q(1)$ if $a$ is a square modulo $q$. Furthermore, an elementary argument shows that $C_q(a)<d(q)\ll_{\epsilon}q^{\epsilon}$ for any $\epsilon>0$, where $d(q)=\sum_{m|q}1$ is the usual divisor function.

To investigate the distribution of the random vector $X_{q;a_1,\dots,a_r}$ we shall first compute its covariance matrix $\text{Cov}_{q;a_1,\dots,a_r}$ (the covariance matrix generalizes the notion of variance to multiple dimensions). Recall that the $j,k$ entry of the covariance matrix corresponds to the covariance between the $j$-th and $k$-th entry of the random vector.
\begin{lem}
The entries of  $\textup{Cov}_{q;a_1,\dots,a_r}$ are

$$\textup{Cov}_{q;a_1,\dots,a_r}(j,k)= \begin{cases}  \Va &\text{ if } j=k \\ B_q(a_j,a_k) &\text{ if } j\neq k,\end{cases}$$
where
$$
\Va:=2\sum_{\substack{\chi\neq \chi_0\\ \chi\bmod q}}\sum_{\gamma_{\chi}>0}\frac{1}{\frac14+\gamma_{\chi}^2},   \textup{ and } B_q(a,b):=\sum_{\substack{\chi\neq \chi_0 \\ \chi\bmod q}}\sum_{\gamma_{\chi}>0}\frac{\chi\left(\frac{b}{a}\right)+\chi\left(\frac{a}{b}\right)}{\frac14 +\gamma_{\chi}^2}.
$$
\end{lem}
\begin{proof}
First, note that $\ex(X(q,a))=-C_q(a)$ since $\ex(U(\gamma_{\chi}))=0$ for all $\gamma_{\chi}$. Therefore,  $\textup{Cov}_{q;a_1,\dots,a_r}(j,k)$ equals
\begin{align*}
&\ex\Big(\big(X(q,a_j)+C_q(a_j)\big)\big(X(q,a_k)+C_q(a_k)\big)\Big)\\
&= \ex\left(\sum_{\substack{\chi\neq \chi_0\\ \chi\bmod q}}\sum_{\gamma_{\chi}>0} \sum_{\substack{\psi\neq \chi_0\\ \psi\bmod q}}\sum_{\widetilde{\gamma}_{\psi}>0} \frac{\big(\chi(a_j)U(\gamma_{\chi})+ \overline{\chi(a_j)U(\gamma_{\chi})}\big)\big(\psi(a_k)U(\widetilde{\gamma}_{\psi})+ \overline{\psi(a_k)U(\widetilde{\gamma}_{\psi})}\big)}{\sqrt{\frac14+\gamma_{\chi}^2}\sqrt{\frac14+\widetilde{\gamma}_{\psi}^2}} \right).
\end{align*}
Since $\ex(U(\gamma_{\chi})U(\widetilde{\gamma}_{\psi}))=0$ for all $\gamma_{\chi},\widetilde{\gamma}_{\psi}$  and
$$\ex\left(U(\gamma_{\chi})\overline{U(\widetilde{\gamma}_{\psi})}\right)=\begin{cases} 1& \text{ if } \chi=\psi \text{ and } \gamma_{\chi}=\widetilde{\gamma}_{\psi}\\
0 & \text{ otherwise},\end{cases}$$
we deduce that
$$ \textup{Cov}_{q;a_1,\dots,a_r}(j,k)= \sum_{\substack{\chi\neq \chi_0 \\ \chi\bmod q}}\sum_{\gamma_{\chi}>0}\frac{\chi\left(a_j/a_k\right)+\chi\left(a_k/a_j\right)}{\frac14 +\gamma_{\chi}^2},$$
which implies the result.
\end{proof}

Our next lemma gives the asymptotic behavior of $\Va$ along with the maximal order of $B_q(a_j,a_k)$. This was established in \cite{La1}, and we should also note that it follows implicitly from the results of \cite{FiM}.
\begin{lem}
Assume GRH. Then
\begin{equation}
 \Va=\phi(q)\log q + O(\phi(q)\log\log q),
\end{equation}
and
\begin{equation}
 \max_{(a,b)\in \mathcal{A}_2(q)}B_q(a,b)\asymp \phi(q).
\end{equation}
\end{lem}
\begin{proof}
First, the asymptotic formula (2.1) is proved in Lemma 3.1 of \cite{La1}.
Now, the fact that $B_q(a_j,a_k)\ll \phi(q)$ is proved in Corollary 5.4 of \cite{La1}, while Proposition 5.1 of \cite{La1} implies $B_q(a,-a)\gg \phi(q)$.
\end{proof}

Here and throughout we shall use the notations $\|\vt\|=\sqrt{\sum_{j=1}^rt_j^2}$ and $|\vt|_{\infty}=\max_{1\leq j\leq r}|t_j|$ for the Euclidean norm and the maximum norm of $\vt\in \mathbb{R}^r$ respectively.
 Our next result is an upper bound for the tail of the distribution $\mu_{q;a_1,\dots,a_r}$. This was established in Proposition 4.1 of \cite{La1} in the case where $r$ is fixed.
\begin{lem}
Let $q$ be large and $2\leq r\leq \phi(q)$ be a positive integer. Then for $R\geq \sqrt{\phi(q)\log q}$ we have
$$\mu_{q;a_1,\dots, a_r}(|\vx|_{\infty}>R)\leq 2r\exp\left(-\frac{R^2}{4\phi(q)\log q}\right),$$
uniformly for all $(a_1,\dots,a_r)\in \mathcal{A}_r(q).$
\end{lem}
\begin{proof}
First, we have
$$ \mu_{q;a_1,\dots, a_r}(|\vx|_{\infty}>R)=\pr(|X_{q;a_1,\dots,a_r}|_{\infty}>R) \leq \sum_{j=1}^r \pr(X(q,a_j)>R)+ \sum_{j=1}^r\pr(X(q,a_j)<-R).
$$
We shall bound only $\pr(X(q,a_j)>R)$, since the corresponding bound for $\pr(X(q,a_j)<-R)$ can be obtained similarly. Let $s>0$ and $(a,q)=1$. Then we have
\begin{align*}
\ex\left(e^{sX(q,a)}\right)&=e^{-sC_q(a)}\prod_{\substack{\chi\neq \chi_0\\\chi\bmod q}}\prod_{\gamma_{\chi}>0}\ex
\left(\frac{2s\re(\chi(a)U(\gamma_{\chi}))}{\sqrt{\frac14+\gamma_{\chi}^2}}\right)\\
&=e^{-sC_q(a)}\prod_{\substack{\chi\neq \chi_0\\\chi\bmod q}}\prod_{\gamma_{\chi}>0}I_0
\left(\frac{2s}{\sqrt{\frac14+\gamma_{\chi}^2}}\right),
\end{align*}
where $I_0(t):= \sum_{n=0}^{
\infty}(t/2)^{2n}/n!^2$ is the modified Bessel function of order $0$.
Hence, using the Chernoff bound along with the fact that $I_0(s)\leq \exp(s^2/4)$ for all $s\in \mathbb{R}$
we derive
$$\pr(X(q,a)>R)\leq e^{-sR} \ex\left(e^{sX(q,a)}\right)\leq \exp\left(-sR-sC_q(a)+\frac{s^2}{2}\Va\right).$$
The lemma follows upon choosing $s=R/(\phi(q)\log q)$,  since $C_q(a)= q^{o(1)}$ and $\Va\sim \phi(q)\log q$ by Lemma 2.2.
\end{proof}
\section{The Fourier transform $\hat{\mu}_{q;a_1,\dots,a_r}$}

Throughout the remaining part of the paper we shall assume both GRH and LI. Moreover, we will use the following normalization for the Fourier transform of an integrable function $f:\mathbb{R}^n\to \mathbb{C}$
$$\hat{f}(t_1,\dots , t_n)=\int_{\mathbb{R}^n}e^{-i(t_1x_1+\cdots+ t_nx_n)}f(x_1,\dots, x_n)dx_1 \dots dx_n.$$
Then if $\hat{f}$ is integrable on $\mathbb{R}^n$ we have the Fourier inversion formula
$$ f(x_1,\dots, x_n)=(2\pi)^{-n}\int_{\mathbb{R}^n}e^{i(t_1x_1+\cdots+ t_nx_n)}\hat{f}(t_1,\dots, t_n)dt_1 \dots dt_n.$$
Similarly we write
$$\hat{\nu}(t_1,\dots, t_n)=\int_{\mathbb{R}^n}e^{-i(t_1x_1+\cdots+ t_nx_n)}d\nu(x_1, \dots, x_n)$$
 for the Fourier transform of a finite measure $\nu$ on $\mathbb{R}^n$.

 Rubinstein and Sarnak \cite{RS} established the following explicit formula for the Fourier transform of $\mu_{q;a_1,\dots,a_r}$
\begin{equation}
\hat{\mu}_{q;a_1,\dots, a_r}(t_1,\dots,t_r)=  \exp\left(i\sum_{j=1}^rC_q(a_j)t_j\right)\prod_{\substack{\chi\neq \chi_0\\ \chi\bmod q}}\prod_{\gamma_{\chi}>0}J_0\left(\frac{2\left|\sum_{j=1}^r\chi(a_j)t_j\right|}
{\sqrt{\frac14+\gamma_{\chi}^2}}\right),
\end{equation}
where $J_0(z)=\sum_{m=0}^{\infty}(-1)^{m}(z/2)^{2m}/m!^2$ is the Bessel function of order $0$.

Our first result shows that in the range $\|\vt\|\leq \Va^{-1/2+o(1)}$, the Fourier transform $\hat{\mu}_{q;a_1,\dots, a_r}(t_1,\dots,t_r)$ is very close to the Fourier transform of a multivariate normal distribution whose covariance matrix equals $\Cov$.

\begin{pro} Let $q$ be large, $2\leq r\leq \log q$ be a positive integer, and $(a_1,\dots,a_r)\in\mathcal{A}_r(q).$ Then in the range $\|\vt\|\leq \Va^{-1/2}\log^2q$ we have
$$ \hat{\mu}_{q;a_1,\dots, a_r}(t_1,\dots,t_r)= \exp\left(-\frac{1}{2}\vt^T\Cov\vt\right)\left(1+O\left(\frac{d(q)\log^3 q}{\sqrt{q}}\right)\right).$$
\end{pro}
\begin{proof} First, the explicit formula (3.1) yields
$$ \log\hat{\mu}_{q;a_1,\dots, a_r}(t_1,\dots,t_r)= \sum_{\substack{\chi\neq \chi_0\\ \chi\bmod q}}\sum_{\gamma_{\chi}>0}\log J_0\left(\frac{2\left|\sum_{j=1}^r\chi(a_j)t_j\right|}
{\sqrt{\frac14+\gamma_{\chi}^2}}\right)+O\left(\|\vt\|\sum_{j=1}^r|C_q(a_j)|\right).$$
Using Lemma 2.2 along with the standard estimate $\phi(q)\gg q/\log\log q$, we deduce that the error term above is $\ll q^{-1/2}d(q)\log^3 q$. On the other hand note that $$\frac{2\left|\sum_{j=1}^r\chi(a_j)t_j\right|}
{\sqrt{\frac14+\gamma_{\chi}^2}}\ll r\|\vt\| \leq 1$$ if $q$ is large enough. Hence, using that $\log J_0(z)=-z^2/4+O(z^4)$ for $|z|\leq 1$ we obtain
\begin{equation}
 \begin{aligned}
 \log\hat{\mu}_{q;a_1,\dots, a_r}(t_1,\dots,t_r)=& -\sum_{\substack{\chi\neq \chi_0\\ \chi\bmod q}}\sum_{\gamma_{\chi}>0} \frac{\left|\sum_{j=1}^r\chi(a_j)t_j\right|^2}{\frac14+\gamma_{\chi}^2}\\
 &+O\left(r^4\|\vt\|^4\sum_{\substack{\chi\neq \chi_0\\ \chi\bmod q}}\sum_{\gamma_{\chi}>0} \frac{1}
{(\frac14+\gamma_{\chi}^2)^2}+\frac{d(q)\log^3 q}{\sqrt{q}}\right).\\
\end{aligned}
\end{equation}
Since $\sum_{\substack{\chi\neq \chi_0\\ \chi\bmod q}}\sum_{\gamma_{\chi}>0}1/
(\frac14+\gamma_{\chi}^2)^2\ll \Va$, it follows that the error term in the above estimate is $\ll q^{-1/2}d(q)\log^3 q.$
On the other hand, the main term on the RHS of (3.2) equals
\begin{align*}
-\sum_{\substack{\chi\neq \chi_0\\ \chi\bmod q}}\sum_{\gamma_{\chi}>0}\frac{1}{\frac14+\gamma_{\chi}^2}\sum_{1\leq j,k\leq r}\chi(a_j)\overline{\chi(a_k)}t_jt_k&=-\frac{1}{2}
\sum_{1\leq j,k\leq r}\Cov(j,k)t_jt_k\\
&= -\frac{1}{2}\vt^T\Cov\vt,
\end{align*}
by Lemma 2.1.
\end{proof}
Next, we show that $\hat{\mu}_{q;a_1,\dots,a_r}(t_1,\dots,t_r)$ is rapidly decreasing in the range $\|\vt\|\geq \Va^{-1/2}.$ In particular, the following result is a refinement of Proposition 3.2 of \cite{La1}, which takes into account the dependence of the upper bounds on $r$.
\begin{pro}  There exists a constant $c_1>0$ such that, if $q$ is large and $2\leq r\leq c_1\log q$, then uniformly for all $(a_1,\dots,a_r)\in\mathcal{A}_r(q)$ we have
$$
|\hat{\mu}_{q;a_1,\dots, a_r}(t_1,\dots,t_r)|\leq \begin{cases} \displaystyle{\exp\left(-\frac{\phi(q)}{8r}\|\vt\|\right)} & \text{ if } \|\vt\|\geq 400,\\
 \displaystyle{\exp\left(-\frac{\phi(q)}{(\log q)^{8}}\right)} &\text{ if } (\log q)^{-2}\leq \|\vt\|\leq 400,\\
 \displaystyle{\exp\left(-\frac{\phi(q)\log q}{4}\|\vt\|^2\right)} &\text{ if }  \|\vt\|\leq (\log q)^{-2}.\end{cases}
 $$
\end{pro}
Before proving this result we first require the following lemma.
\begin{lem} Let $q$ be large and $2\leq r\leq \phi(q)/4$ be an integer. For $\va=(a_1,\dots,a_r)\in \mathcal{A}_r(q)$ and $\vt\in \mathbb{R}^r$ we denote by  $M_{q,\va}(\vt)$ the set of nontrivial characters $\chi\bmod q$ such that
 $\left|\sum_{j=1}^r\chi(a_j)t_j\right|\geq \|\vt\|/2$. Then
$$
|M_{q,\va}(\vt)|\geq \frac{\phi(q)}{2 r}.
$$
\end{lem}
\begin{proof} 
 Let
\begin{equation}
\begin{aligned}
 S(\vt)&=\sum_{\substack{\chi\neq \chi_0 \\ \chi\bmod q}}\left|\sum_{j=1}^r\chi(a_j)t_j\right|^2= \sum_{\chi\bmod q}\left|\sum_{j=1}^r\chi(a_j)t_j\right|^2-\left(\sum_{j=1}^rt_j\right)^2\\
&= \sum_{j=1}^r\sum_{k=1}^rt_jt_k\sum_{\chi\bmod q}\chi(a_j)\overline{\chi(a_k)}-\left(\sum_{j=1}^rt_j\right)^2=\phi(q)\sum_{j=1}^rt_j^2-\left(\sum_{j=1}^rt_j\right)^2\\&\geq (\phi(q)-r)\|\vt\|^2,
\end{aligned}
\end{equation}
by the Cauchy-Schwarz inequality. Therefore, using that
$\left|\sum_{j=1}^r\chi(a_j)t_j\right|^2\leq \left(\sum_{j=1}^r|t_j|\right)^2\leq r\|\vt\|^2$, we deduce
$$
 S(\vt)=\sum_{\chi\in M_{q,\va}(\vt)}\left|\sum_{j=1}^r\chi(a_j)t_j\right|^2+ \sum_{\chi\notin M_{q,\va}(\vt)}\left|\sum_{j=1}^r\chi(a_j)t_j\right|^2
\leq r|M_{q,\va}(\vt)|\|\vt\|^2+ \frac{\phi(q)}{4}\|\vt\|^2.
$$
Combining this estimate with (3.3) completes the proof.
\end{proof}

\begin{proof}[Proof of Proposition 3.2]
First, assume that $\|\vt\|\geq 400$. For any nontrivial character $\chi\bmod q$ we define
$$ F(x,\chi):=\prod_{\gamma_{\chi}>0}J_0\left(\frac{2x}{\sqrt{\frac14+\gamma_{\chi}^2}}\right).$$
Then, it follows from Lemma 2.16 of \cite{FiM} that
\begin{equation}
|F(x,\chi)F(x,\overline{\chi})|\leq e^{-x}
\end{equation} for $x\geq 200$. Moreover, the explicit formula (3.1) implies 
$$|\hat{\mu}_{q;a_1,\dots, a_r}(t_1,\dots,t_r)|= \prod_{\substack{\chi\neq \chi_0\\ \chi\bmod q}}\left|F\left(\left|\sum_{j=1}^r\chi(a_j)t_j\right|,\chi\right)\right|.$$ 
If $\chi \in M_{q,\va}(\vt)$  then $\left|\sum_{j=1}^r\chi(a_j)t_j\right|\geq 200$. Furthermore, note that $\chi \in M_{q,\va}(\vt)$ if and only if $\overline{\chi}\in M_{q,\va}(\vt)$. Hence, using  (3.4) along with the trivial bound $|F(x,\chi)|\leq 1$ (since $|J_0(x)|\leq 1$) we derive
\begin{equation}
\begin{aligned}
|\hat{\mu}_{q;a_1,\dots, a_r}(t_1,\dots,t_r)|^2&\leq \prod_{\chi\in M_{q,\va}(\vt)}\left|F\left(\left|\sum_{j=1}^r\chi(a_j)t_j\right|,\chi\right)\right|^2\\
&= \prod_{\chi\in M_{q,\va}(\vt)}\left|F\left(\left|\sum_{j=1}^r\chi(a_j)t_j\right|,\chi\right)F\left(\left|\sum_{j=1}^r\chi(a_j)t_j\right|,\overline{\chi}\right)\right|\\
&\leq \exp\left(-\sum_{\chi\in M_{q,\va}(\vt)}\left|\sum_{j=1}^r\chi(a_j)t_j\right|\right)\leq \exp\left(-\frac{1}{2}|M_{q,\va}(\vt)\||\vt\|\right),
\end{aligned}
\end{equation}
since every character in $M_{q,\va}(\vt)$ appears once as $\chi$ and once as $\overline{\chi}$ in the product on the RHS of (3.7).
Combining this inequality with Lemma 3.3 yield the desired bound on $|\hat{\mu}_{q;a_1,\dots, a_r}(t_1,\dots,t_r)|$ in this case.

Let $\epsilon=(\log q)^{-2}$ and suppose that $\epsilon\leq \|t\|\leq 400$.  If $\chi\in M_{q,\va}(\vt)$ then
$$\frac{2\left|\sum_{j=1}^r\chi(a_j)t_j\right|}{\sqrt{\frac14+\gamma_{\chi}^2}}\geq \frac{\epsilon}{\sqrt{\frac14+\gamma_{\chi}^2}}. $$  We also note that if $q$ is sufficiently large  then $\epsilon\left(\frac14+\gamma_{\chi}^2\right)^{-1/2}\leq 2\epsilon\leq 1$. Therefore, since $J_0$ is a positive decreasing function on $[0,1]$ and $|J_0(z)|\leq J_0(1)$ for all $z\geq 1$, we get
$$ |\hat{\mu}_{q;a_1,\dots, a_r}(t_1,\dots,t_r)|\leq \prod_{\chi\in M_{q,\va}(\vt)}\prod_{\gamma_{\chi}>0}\left|J_0\left(\frac{\epsilon}{\sqrt{\frac14+\gamma_{\chi}^2}}\right)\right|.$$
Furthermore, using the standard bound $|J_0(x)|\leq \exp(-x^2/4)$ for $|x|\leq 1$, we deduce that
\begin{equation}
 |\hat{\mu}_{q;a_1,\dots, a_r}(t_1,\dots,t_r)|\leq \exp\left(-\frac{\epsilon^2}{4}\sum_{\chi\in M_{q,\va}(\vt)}\sum_{ \gamma_{\chi}>0}\frac{1}{\frac14+\gamma_{\chi}^2}\right).
\end{equation}
Let $N(T,\chi)$ denote the number of $\gamma_{\chi}$ in the interval $[0,T]$. Then, we have the classical estimate (see Chapiters 15 and 16 of \cite{Da}) $$N(T,\chi)=\frac{T}{2\pi}\log \frac{q^*T}{2\pi e}+O(\log q T),$$
where $q^*$ is the conductor of $\chi$. Hence, if $T=\log ^2q$ then $N(T,\chi)\gg \log^2q$. This yields
$$ \sum_{ \gamma_{\chi}>0}\frac{1}{\frac14+\gamma_{\chi}^2}\geq \sum_{ 0<\gamma_{\chi}\leq \log^2q}\frac{1}{\frac14+\gamma_{\chi}^2}\gg \frac{1}{\log^2q}.$$
The upper bound on $|\hat{\mu}_{q;a_1,\dots, a_r}(t_1,\dots,t_r)|$ then follows upon inserting this estimate in (3.6) and using Lemma 3.3.

Finally assume that $\|\vt\|\leq (\log q)^{-2}.$ If $q$ is large enough then
$$ \frac{2\left|\sum_{j=1}^r\chi(a_j)t_j\right|}
{\sqrt{\frac14+\gamma_{\chi}^2}}\ll r\|\vt\| \leq 1.$$ Hence, using that $|J_0(x)|\leq \exp(-x^2/4)$ for $|x|\leq 1$ we obtain from the explicit formula (3.1)
\begin{equation}
|\hat{\mu}_{q;a_1,\dots, a_r}(t_1,\dots,t_r)|\leq \exp\left(-\sum_{\substack{\chi\neq \chi_0\\ \chi\bmod q}}\sum_{\gamma_{\chi}>0}\frac{\left|\sum_{j=1}^r\chi(a_j)t_j\right|^2}{\frac14+\gamma_{\chi}^2}\right).
\end{equation}
Furthermore, Lemma 2.2 yields
\begin{align*}
\sum_{\substack{\chi\neq \chi_0\\ \chi\bmod q}}\sum_{\gamma_{\chi}>0}\frac{\left|\sum_{j=1}^r\chi(a_j)t_j\right|^2}{\frac14+\gamma_{\chi}^2}
&= \sum_{\substack{\chi\neq \chi_0\\ \chi\bmod q}}\sum_{\gamma_{\chi}>0}\frac{1}{\frac14+\gamma_{\chi}^2}\sum_{1\leq j,k\leq r}\chi(a_j)\overline{\chi(a_k)}t_jt_k\\
&= \frac{\Va}{2}\left(t_1^2+\cdots+t_r^2\right)+ \sum_{1\leq j<k\leq r} B_q(a_j,a_k)t_jt_k\\
&= \frac{\phi(q)\log q}{2}\|\vt\|^2\left(1+ O\left(\frac{r+\log\log q}{\log q}\right)\right),
\end{align*}
since
$$ \sum_{1\leq j<k\leq r} |t_jt_k|\leq \left(\sum_{j=1}^r|t_j|\right)^2\leq r\|\vt\|^2,$$ by the Cauchy-Schwarz inequality. Thus, if $r\leq c_1\log q$ where $c_1>0$ is suitably small, then
$$\sum_{\substack{\chi\neq \chi_0\\ \chi\bmod q}}\sum_{\gamma_{\chi}>0}\frac{\left|\sum_{j=1}^r\chi(a_j)t_j\right|^2}{\frac14+\gamma_{\chi}^2}\geq \frac{\phi(q)\log q}{4}\|\vt\|^2.$$
Inserting this estimate in (3.7) completes the proof.
\end{proof}

\section{The asymptotic behavior of the densities $\delta_{q; a_1,\dots,a_r}$: Proof of Theorems 1.1, 1.2 and 1.3}

 We showed in the previous section that in a small region around $0$, the Fourier transform of $\mu_{q;a_1,\dots, a_r}$ can be approximated by the Fourier transform of a multivariate normal distribution whose covariance matrix equals $\Cov$. If we normalize by $\sqrt{\Va}$ then Proposition 3.1 above implies that in the range $\|\vt\|\leq \log^2 q$ we have
\begin{equation}
\hat{\mu}_{q;a_1,\dots,a_r}\left(\frac{t_1}{\sqrt{\Va}}, \dots,\frac{t_r}{\sqrt{\Va}}\right)=\exp\left(-\frac{1}{2}\vt^T\mathcal{C}\vt\right)\left(1+O\left(\frac{d(q)\log^3 q}{\sqrt{q}}\right)\right),
\end{equation}
where $\mathcal{C}$ is an $r\times r$ symmetric matrix whose entries are
$$ \mathcal{C}_{jk}= \begin{cases} 1 & \text{ if } j=k,\\
\displaystyle{\frac{B_q(a_j,a_k)}{\Va}}\ll \frac{1}{\log q} & \text{ if } j\neq k.
\end{cases}
$$

Let $\mathcal{M}_r(\epsilon)$ denote the set of $r\times r$ symmetric matrices $A=(a_{jk})$ such that $a_{jj}=1$ for all $1\leq j\leq r$ and $|a_{jk}|\leq \epsilon$ for all $1\leq j\neq k\leq r$.
In order to prove Theorems 1.1-1.3, we need to investigate multivariate normal distributions whose covariance matrices belong to $\mathcal{M}_r(\epsilon)$ where $\epsilon\ll 1/\log q$ is small. To this end we shall study the density function of a multivariate normal distribution, which is given by
\begin{equation}
 f(\vx)= \frac{1}{(2\pi)^{r/2}\det(A)}\exp\left(-\frac12\vx^T A^{-1}\vx\right),
\end{equation}
if $A$ is the covariance matrix of the distribution.

Our first lemma shows that the determinant of any matrix $A\in \mathcal{M}_r(\epsilon)$ is close to $1$ if $\epsilon$ is small enough.
\begin{lem} If $\epsilon\leq 1/(2r)$ then for any $A\in \mathcal{M}_r(\epsilon)$ we have $\det(A)=1+O(\epsilon^2 r^2).$
\end{lem}
\begin{proof} Let $S_r$ be the set of all permutations $\sigma$ of $\{1,\dots, r\}$. Then we have
\begin{equation}
\det(A)= \sum_{\sigma\in S_r}\text{sgn}(\sigma)a_{1\sigma(1)}\cdots a_{r\sigma(r)}= 1+ \sum_{\substack{\sigma\in S_r\\ \sigma\neq \mathbf{1}}}\text{sgn}(\sigma)a_{1\sigma(1)}\cdots a_{r\sigma(r)},
\end{equation}
where $\mathbf{1}$ denotes the identity permutation. For $0\leq k\leq r$ let $S_r(k)$ be the set of permutations $\sigma\in S_r$ such that the equation $\sigma(j)=j$ has exactly $r-k$ solutions in $\{1,\dots,r\}$. Then $S_r(0)=\{\mathbf{1}\}$, $S_r(1)=\emptyset$ and more generally one has
 $$|S_r(k)|\leq \binom{r}{r-k} (k-1)!\leq r^k, \text{ for } 2\leq k\leq r.$$
Moreover, note that $|a_{1\sigma(1)}\cdots a_{r\sigma(r)}|\leq \epsilon^k$, for all $\sigma\in S_r(k).$

Hence, we deduce
$$ \sum_{\substack{\sigma\in S_r\\ \sigma\neq \mathbf{1}}}\text{sgn}(\sigma)a_{1\sigma(1)}\cdots a_{r\sigma(r)}= \sum_{k=2}^r \sum_{\sigma\in S_r(k)}\text{sgn}(\sigma)a_{1\sigma(1)}\cdots a_{r\sigma(r)}\ll \sum_{k=2}^r (\epsilon r)^k \ll \epsilon^2 r^2.$$
Inserting this estimate in (4.3) implies the result.
\end{proof}
In order to understand the behavior of the density function $f(\vx)$ we need to determine the size of the entries $\{\tilde{a}_{jk}\}$ of $A^{-1}$, if $A\in \mathcal{M}_r(\epsilon)$.  The next lemma shows that if $\epsilon$ is small then the diagonal entries are close to $1$ and the off-diagonal ones are small.
\begin{lem} If $\epsilon\leq 1/(2r)$ then for any $A\in \mathcal{M}_r(\epsilon)$ we have
$$ \tilde{a}_{jk}=\begin{cases} 1+O(\epsilon^2 r^2) & \text{ if } j=k, \\
O(\epsilon) & \text{ if } j\neq k.\\
\end{cases}
$$

\end{lem}
\begin{proof} Recall that
$$\tilde{a}_{jk}= \frac{1}{\det(A)}(-1)^{j+k}M_{kj},$$
where $M_{kj}$ is the minor of the entry $a_{kj}$ which is given by $M_{kj}=\det(A_{kj})$ and $A_{kj}$ is the matrix obtained from $A$ by deleting the $k$-th row and the $j$-th column.

First, we determine the size of the diagonal entries $\tilde{a}_{jj}.$ In this case, remark that $A_{jj}\in \mathcal{M}_{r-1}(\epsilon).$ Hence, it follows from Lemma 4.1 that
$$\tilde{a}_{jj}= \frac{\det(A_{jj})}{\det(A)}=1+O(\epsilon^2r^2).$$
Now, we handle the off-diagonal entries. For $1\leq j\neq k\leq r$, let $\mathcal{B}_{j,k}$ denote the set of all bijections $\sigma$ from $\{1,\dots,r\}\setminus\{j\}$ to $\{1,\dots,r\}\setminus\{k\}$. Then, we have
$$ |M_{jk}|=|\det(A_{jk})|\leq \sum_{\sigma\in \mathcal{B}_{j,k}}\prod_{1\leq n\neq j\leq r}|a_{n\sigma(n)}|.$$
For $0\leq l\leq r-1$ we define $\mathcal{B}_{j,k}(l)$ to be the set of bijections $\sigma\in \mathcal{B}_{j,k}$ such that the equation $\sigma(m)=m$ has
 exactly $r-1-l$ solutions. Since $\sigma(k)\neq k$ then it follows that $\mathcal{B}_{j,k}(0)=\emptyset$, and more generally one has
 $$|\mathcal{B}_{j,k}(l)|\leq \binom{r-2}{r-1-l}(l-1)! \leq r^{l-1}, \text{ for } 1\leq l\leq r-1.$$
 Hence we obtain
 $$ |M_{jk}|\leq \sum_{l=1}^{r-1} \sum_{\sigma\in \mathcal{B}_{j,k}(l)}\prod_{1\leq n\neq j\leq r}|a_{n\sigma(n)}|
 \ll \sum_{l=1}^{r-1} r^{l-1}\epsilon^l\ll\epsilon.$$
Combining this bound with Lemma 4.1 yield the desired bound
 $ \tilde{a}_{jk}\ll \epsilon$.
\end{proof}

We know that the Fourier transform of a multivariate Gaussian of covariance matrix $A$ is (up to normalization) a multivariate Gaussian of covariance $A^{-1}$. The last ingredient we need to prove Theorems 1.1-1.3 is  an approximate version of this statement when $A\in \mathcal{M}_r(\epsilon)$.
\begin{lem} Let $r\geq 2$ be a positive integer, $R\geq 10\sqrt{r}$ be a real number  and $\vx\in \mathbb{R}^r$. If $\epsilon\leq 1/(2r)$ then for any $A\in \mathcal{M}_r(\epsilon)$ we have
 \begin{align*}
 (2\pi)^{-r}\int_{\|\vt\|\leq R}e^{i(t_1x_1+\cdots+t_rx_r)}\exp\left(-\frac12\vt^T A\vt\right)d\vt=& \frac{1}{(2\pi)^{r/2}\det(A)}\exp\left(-\frac12\vx^T A^{-1}\vx\right)\\
 &+O\left(\exp\left(-\frac{R^2}{5}\right)\right).\\
 \end{align*}
\end{lem}
\begin{proof}
Since $\exp\left(-\frac12\vt^TA\vt\right)$ is the Fourier transform of the multivariate normal distribution whose density equals
$$f(\vx)= \frac{1}{(2\pi)^{r/2}\det(A)}\exp\left(-\frac12\vx^T A^{-1}\vx\right), $$
then the Fourier inversion formula yields
\begin{equation}
 (2\pi)^{-r}\int_{\vt\in\mathbb{R}^r}e^{i(t_1x_1+\cdots+t_rx_r)}\exp\left(-\frac12\vt^T A \vt\right)d\vt= \frac{1}{(2\pi)^{r/2}\det(A)}\exp\left(-\frac12\vx^T A^{-1}\vx\right).
\end{equation}
Moreover, since $|a_{jk}|\leq 1/(2r)$ for $j\neq k$ then
$$\left|\sum_{1\leq j\neq k\leq r}a_{jk}t_jt_k\right|\leq \frac{1}{2r} \left(\sum_{j=1}^r|t_j|\right)^2\leq \frac{1}{2}\sum_{j=1}^rt_j^2,$$
by the Cauchy-Schwarz inequality. This implies
\begin{equation}
 \vt^T A\vt =\sum_{j=1}^r\sum_{k=1}^ra_{jk}t_jt_k\geq \frac{1}{2}\sum_{j=1}^rt_j^2.
\end{equation}
Hence, we get
$$(2\pi)^{-r}\int_{\|\vt\|> R}\exp\left(-\frac12\vt^T A\vt\right)d\vt\leq (2\pi)^{-r}\int_{\|\vt\|> R}\exp\left(-\frac14\|\vt\|^2\right)d\vt\ll\exp\left(-\frac{R^2}{5}\right),$$
which in view of (4.4) completes the proof.
\end{proof}

\begin{proof}[Proof of Theorem 1.1]
 To lighten the notation we shall write $\delta_q$ for $\delta_{q;a_1,\dots, a_r}$ and $\mu_q$ for $\mu_{q;a_1,\dots,a_r}$. Let $R=3\sqrt{\Va}\log q.$ First, using Lemma 2.3 we derive
\begin{equation}
\delta_q= \int_{y_1>y_2>\dots>y_r}d\mu_q(y_1, \dots, y_r)= \int_{\substack{y_1>y_2>\dots>y_r\\ |\mathbf{y}|_{\infty}\leq R}}d\mu_q(y_1, \dots, y_r)+ O\left(\exp\left(-2\log^2 q\right)\right).
\end{equation}
Next, we apply the Fourier inversion formula to the measure $\mu_q$ to get
$$ \int_{\substack{y_1>y_2>\dots>y_r\\ |\mathbf{y}|_{\infty}\leq R}}d\mu_q(y_1, \dots, y_r)
=(2\pi)^{-r}\int_{\substack{y_1>y_2>\dots>y_r\\ |\mathbf{y}|_{\infty}\leq R}}\int_{\mathbf{s}\in \mathbb{R}^r}e^{i(s_1y_1+\cdots+s_ry_r)}\hat{\mu}_q(s_1,\dots, s_r)d\mathbf{s}d\mathbf{y}.$$
Since the Fourier transform $\hat{\mu}_q(s_1,\dots, s_r)$ is rapidly decreasing, we shall deduce that the main contribution to the integral over $\mathbb{R}^r$ of $e^{i(s_1y_1+\cdots+s_ry_r)}\hat{\mu}_q(s_1,\dots, s_r)$ comes from a small ball centered at $0$. Indeed, we infer from Proposition 3.2 that
\begin{align*}
\int_{\mathbf{s}\in \mathbb{R}^r}e^{i(s_1y_1+\cdots+s_ry_r)}\hat{\mu}_q(s_1,\dots, s_r)d\mathbf{s}
=& \int_{\|\mathbf{s}\|\leq \epsilon}e^{i(s_1y_1+\cdots+s_ry_r)}\hat{\mu}_q(s_1,\dots, s_r)d\mathbf{s}\\
 &+ O\left(\exp\left(-2\log^2 q\right)\right),
\end{align*} where
$\epsilon= 3(\Va)^{-1/2}\log q.$
Hence we get
\begin{equation}\delta_q= (2\pi)^{-r}\int_{\substack{y_1>y_2>\dots>y_r\\ |\mathbf{y}|_{\infty}\leq R}}\int_{\|\mathbf{s}\|\leq \epsilon}e^{i(s_1y_1+\cdots+s_ry_r)}\hat{\mu}_q(s_1,\dots, s_r)d\mathbf{s}d\mathbf{y}+ O\left(\exp\left(-\log^2 q\right)\right),
\end{equation}
since $R^r\ll \exp(r\log q).$ Now, we make the change of variables
$$ t_j:=  \sqrt{\Va} s_j \text{ and } x_j:= \frac{y_j}{\sqrt{\Va}}, \text{ for all } 1\leq j\leq r$$
to obtain
\begin{equation}
 \begin{aligned}
 \delta_q=& (2\pi)^{-r}\int_{\substack{x_1>x_2>\dots>x_r\\ |\vx|_{\infty}\leq 3\log q}}\int_{\|\vt\| \leq 3\log q }e^{i(t_1x_1+\cdots+t_rx_r)}\hat{\mu}_q\left(\frac{t_1}{\sqrt{\Va}},\dots, \frac{t_r}{\sqrt{\Va}}\right)d\vt d\vx\\
 &+ O\left(\exp\left(-\log^2 q\right)\right).\\
 \end{aligned}
\end{equation}
Replacing $\hat{\mu}_q\left(\frac{t_1}{\sqrt{\Va}},\dots, \frac{t_r}{\sqrt{\Va}}\right)$ by the approximation (4.1) that we derived in Proposition 3.1 yields
$$ \delta_q= (2\pi)^{-r}\int_{\substack{x_1>x_2>\dots>x_r\\ |\vx|_{\infty}\leq 3\log q}}\int_{\|\vt\| \leq 3\log q }e^{i(t_1x_1+\cdots+t_rx_r)}\exp\left(-\frac{1}{2}\vt^T\mathcal{C}\vt\right)d\vt d\vx+ E_1,$$
where
$$ E_1\ll q^{-1/3}(\log q)^{3r}\ll q^{-1/4},$$
since $d(q)=q^{o(1)}$ and $\vt^T\mathcal{C}\vt\geq 0$ by (4.5).
Furthermore, applying Lemma 4.3 we derive
\begin{equation}
\delta_q= \frac{1}{(2\pi)^{r/2}\det(\mathcal{C})}\int_{\substack{x_1>x_2>\dots>x_r\\ |\vx|_{\infty}\leq 3\log q}}\exp\left(-\frac12\vx^T \mathcal{C}^{-1}\vx\right)d\vx+ O\left(q^{-1/4}\right).
\end{equation}
Since $\mathcal{C}_{jk}=B_q(a_j,a_k)/\Va\ll (\log q)^{-1}$ for $j\neq k$ by Lemma 2.2, then there exists an absolute constant $\alpha_0>0$ such that $\mathcal{C}\in \mathcal{M}_r(\beta)$ with $\beta=\alpha_0/\log q.$ Therefore, appealing to Lemma 4.2 we obtain
\begin{align*}
 \vx^T \mathcal{C}^{-1}\vx&= \left(1+O\left(\frac{r^2}{\log^2q}\right)\right)\sum_{j=1}^r x_j^2+ O\left(\frac{1}{\log q}\left(\sum_{j=1}^r|x_j|\right)^2\right)\\
 &= \left(1+O\left(\frac{r}{\log q}\right)\right)\|\vx\|^2,\\
\end{align*}
which follows from the Cauchy-Schwarz inequality. Hence we deduce
\begin{equation}
 -\frac12\left(1+\frac{\alpha_1r}{\log q}\right)\|\vx\|^2\leq -\frac12\vx^T \mathcal{C}^{-1}\vx
 \leq -\frac12\left(1-\frac{\alpha_1r}{\log q}\right)\|\vx\|^2,
\end{equation}
for some absolute constant $\alpha_1>0$. This implies
$$ \int_{\substack{x_1>x_2>\dots>x_r\\ |\vx|_{\infty}> 3\log q}}\exp\left(-\frac12\vx^T \mathcal{C}^{-1}\vx\right)d\vx\leq
\int_{ |\vx|_{\infty}>3\log q}\exp\left(-\frac14\|\vx\|^2\right)d\vx \ll \exp\left(-\log^2 q\right).$$
Inserting this estimate in (4.9) and using Lemma 4.1 we get
\begin{equation}
 \delta_q= \left(1+O\left(\frac{r^2}{\log^2q}\right)\right)\frac{1}{(2\pi)^{r/2}}\int_{x_1>x_2>\dots>x_r}\exp\left(-\frac12\vx^T \mathcal{C}^{-1}\vx\right)d\vx+ O\left(q^{-1/4}\right).
\end{equation}
Let $\kappa$ be a real number such that $|\kappa|\leq \alpha_1r/\log q$. Since the function $\|\vx\|^2$ is symmetric in the variables $\{x_j\}_{1\leq j\leq r}$ we obtain
\begin{equation}
\begin{aligned}
\frac{1}{(2\pi)^{r/2}}\int_{x_1>x_2>\dots>x_r}\exp\left(-\frac12(1+\kappa)\|\vx\|^2\right)d\vx&=\frac{1}{r!}\left(\frac{1}{\sqrt{2\pi}}\int_{-\infty}^{\infty}
\exp\left(-\frac12(1+\kappa)y^2\right)dy\right)^r\\
&= \frac{1}{r!(1+\kappa)^{r/2}}=\frac{1}{r!}\exp\left(O\left(\frac{r^2}{\log q}\right)\right).
\end{aligned}
\end{equation}
The theorem follows upon combining this estimate with (4.10) and (4.11).
\end{proof}
\begin{proof}[Proof of Theorem 1.2]
The result can be obtained by proceeding along the same lines as the proof of Theorem 1.1, except that we make a different choice of parameters in this case. Indeed, choosing $R=5\sqrt{\Va r\log r}$ and using Lemma 2.3 and Proposition 3.2, we obtain analogously to (4.8)
\begin{equation}
 \begin{aligned}
 \delta_q=& (2\pi)^{-r}\int_{\substack{x_1>x_2>\dots>x_r\\ |\vx|_{\infty}\leq 5\sqrt{r\log r}}}\int_{\|\vt\| \leq 3\log q }e^{i(t_1x_1+\cdots+t_rx_r)}\hat{\mu}_q\left(\frac{t_1}{\sqrt{\Va}},\dots, \frac{t_r}{\sqrt{\Va}}\right)d\vt d\vx\\
 &+ O\left(\exp\left(-4r\log r\right)\right).\\
 \end{aligned}
\end{equation}
Moreover, we infer from (4.1) that
\begin{equation}
 \delta_q= (2\pi)^{-r}\int_{\substack{x_1>x_2>\dots>x_r\\ |\vx|_{\infty}\leq 5\sqrt{r\log r}}}\int_{\|\vt\| \leq 3\log q }e^{i(t_1x_1+\cdots+t_rx_r)}\exp\left(-\frac{1}{2}\vt^T\mathcal{C}\vt\right)d\vt d\vx+ E_2,
\end{equation}
where
\begin{equation}
E_2\ll \frac{d(q)\log^3q}{\sqrt{q}}(2\pi)^{-r}\int_{\substack{x_1>x_2>\dots>x_r\\ |\vx|_{\infty}\leq 5\sqrt{r\log r}}}d\vx \int_{\|\vt\| \leq 3\log q }\exp\left(-\frac{1}{2}\vt^T\mathcal{C}\vt\right)d\vt + \exp\left(-4r\log r\right).
\end{equation}
Note that
$$\int_{\substack{x_1>x_2>\dots>x_r\\ |\vx|_{\infty}\leq 5\sqrt{r\log r}}}d\vx= \frac{1}{r!}\int_{|\vx|_{\infty}\leq 5\sqrt{r\log r}}d\vx=\frac{(10\sqrt{r\log r})^r}{r!}=\exp\left(-\frac{r\log r}{2}+O(r\log\log r)\right),$$
by Stirling's formula.
On the other hand, it follows from (4.5) that
$$ \frac{1}{(2\pi)^r}\int_{\|\vt\| \leq 3\log q }\exp\left(-\frac{1}{2}\vt^T\mathcal{C}\vt\right)d\vt\leq \frac{1}{(2\pi)^r} \int_{\vt\in \mathbb{R}^r} \exp\left(-\frac{\|\vt\|^2}{4}\right)d\vt= \frac{1}{\pi^{r/2}}.$$
Therefore, inserting these estimates in (4.15) and using the classical bound $d(q)=\exp\left(O(\log q/\log\log q)\right)$ we deduce
$$
E_2\ll \exp\left(-\frac{1}{2}(\log q+r\log r) +O\left(\frac{\log q}{\log\log q}+r\log\log r\right)\right)+ \exp\left(-4r\log r\right).
$$
Continuing along the same line as in the proof of Theorem 1.1 we obtain analogously to (4.11)
\begin{equation}
 \delta_q= \left(1+O\left(\frac{r^2}{\log^2q}\right)\right)\frac{1}{(2\pi)^{r/2}}\int_{x_1>x_2>\dots>x_r}\exp\left(-\frac12\vx^T \mathcal{C}^{-1}\vx\right)d\vx+ E_3,
\end{equation}
where
$$ E_3\ll \exp\left(-\frac{1}{2}(\log q+r\log r) +O\left(\frac{\log q}{\log\log q}+r\log\log r\right)\right)+ \exp\left(-4r\log r\right).
$$
Furthermore, it follows from (4.10) and (4.12) that
\begin{align*}
 \frac{1}{(2\pi)^{r/2}}\int_{x_1>x_2>\dots>x_r}\exp\left(-\frac12\vx^T \mathcal{C}^{-1}\vx\right)d\vx&= \frac{1}{r!}\exp\left(O\left(\frac{r^2}{\log q}\right)\right)\\
 &=\exp\left(-r\log r+r+O\left(\log r+ \frac{r^2}{\log q}\right)\right),\\
\end{align*}
by Stirling's formula.
Inserting this estimate in (4.16) completes the proof.

\end{proof}
\begin{proof}[Proof of Theorem 1.3]
Since $\mu_{q;a_1,\dots,a_r}$ is absolutely continuous with respect to the Lebesgue measure, it follows from (1.1) that
$$ \delta_{q;a_1,\dots,a_{r-1}}= \delta_{q;a_r,a_1,\dots,a_{r-1}}+ \delta_{q;a_1,a_r,\dots,a_{r-1}}+ \cdots+ \delta_{q;a_1,\dots,a_{r-1},a_{r}}.$$ Hence, if $2\leq s<r\leq \phi(q)$ are positive integers then
\begin{equation}
\max_{(a_1, \dots, a_r)\in \mathcal{A}_r(q)}\delta_{q;a_1,\dots,a_r} < \max_{(b_1, \dots, b_s)\in \mathcal{A}_s(q)}\delta_{q;b_1,\dots,b_s}.
\end{equation}
On the other hand, using Theorem 1.2 with $s=[(1-\epsilon/2)\log q/\log\log q]$, we get
$$ \max_{(b_1, \dots, b_s)\in \mathcal{A}_s(q)}\delta_{q;b_1,\dots,b_s}= \exp\left(-s\log s+s+O\left(\log s+ \frac{r^2}{\log q}\right)\right)\ll_{\epsilon} \frac{1}{q^{1-\epsilon}}.$$
The theorem follows upon combining this inequality with (4.17).
\end{proof}

\end{document}